\documentclass[12pt]{article}
 \usepackage[margin=1in]{geometry} 
\usepackage{amsmath,amsthm,amssymb,amsfonts}
 \usepackage{amssymb}

\usepackage[utf8]{inputenc}

\newcounter{cnstcnt}

\usepackage{hyperref}
 \usepackage{csquotes}
\usepackage[utf8]{inputenc}
\usepackage[english]{babel}
\usepackage{xcolor}
\usepackage{biblatex}
\title{Blow-up prevention by sub-logistic sources in Keller-Segel cross diffusion type system }
\author{Minh Le }
\date{\today}

\begin{document}
\maketitle
\begin{abstract}
The focus of this paper is on solutions to a two-dimensional Keller-Segel system containing sub-logistic sources. We show that the presence of sub-logistic terms is adequate to prevent blow-up phenomena even in strongly degenerate Keller-Segel systems. Our proof relies on several techniques, including parabolic regularity theory in Orlicz spaces, variational arguments, interpolation inequalities, and the Moser iteration method. 
\end{abstract}
\numberwithin{equation}{section}
\newtheorem{theorem}{Theorem}[section]
\newtheorem{lemma}[theorem]{Lemma}
\newtheorem{remark}{Remark}[section]
\newtheorem{Prop}{Proposition}[section]
\newtheorem{Def}{Definition}[section]
\newtheorem{Corollary}{Corollary}[theorem]
\allowdisplaybreaks
\section{Introduction} \label{introduction}
We consider the following nonlinear parabolic cross-diffusion partial differential equations arises from chemotaxis models
\begin{equation} \tag{KS} \label{general-log-1}
    \begin{cases}
        u_t = \nabla \cdot(D(v) \nabla u) - \nabla \cdot (uS(v)\nabla v) +f(u)\\
         v_t = \Delta v -v +u
    \end{cases}
\end{equation}
in a bounded domain $\Omega \subset \mathbb{R}^2$ with smooth boundary, where
\begin{align} \label{diffusion-condition}
    0<D \in C^2[0,\infty) \quad \text{ and } S \in C^2([0,\infty))\cap W^{1,\infty} ((0,\infty)) \text{ such that } S' \geq 0,
\end{align} and $f$ is a smooth function generalizing the sub-logistic and signal production source respectively,  
\begin{align} \label{logistic}
    f(u) = ru-\mu \frac{u^2}{\ln^p(u+e)} ,\quad \text{with } r\in \mathbb{R}, \mu>0, \text{and } p>0,
\end{align}
The system \eqref{general-log-1} is complemented with nonnegative initial conditions in $W^{1,\infty}(\Omega)$ not identically zero:
\begin{align} \label{initial-data}
    u(x,0)=u_0(x), \qquad v(x,0)=v_0(x), \qquad \text{with } x\in \mathbb{R},
 \end{align}
and homogeneous Neumann boundary condition are imposed as follows:
\begin{equation} \label{boundary-data}
    \frac{\partial u}{\partial \nu } = \frac{\partial v}{\partial \nu }  = 0, \qquad x\in \partial \Omega,\, t \in (0, T_{\rm max}),
\end{equation}
where $\nu$ denotes the outward normal vector. Through out this paper, notations $c$ and $\epsilon$, unless being specified, present for positive constants depending on various parameters and sufficiently small positive numbers respectively. \\
Chemotaxis, which describes the movement of a cell toward a chemical signal, has been extensively studied, particularly from a mathematical perspective. One of the first PDE models for chemotaxis, such as \eqref{general-log-1}, was introduced in \cite{Keller}. Interested readers are referred to \cite{Horstmann, Hillen} to learn more about the derivation of the system \eqref{general-log-1} and its biological background. Moreover, \cite{Winkler-2015} summarizes many important results and fundamental analysis techniques. Among the many forms studied, the simplified form in two spatial dimensions, with $S \equiv D \equiv 1$ and $f \equiv 0$, has been extensively investigated in various directions, such as global existence and qualitative behavior of solutions in \cite{Dolbeault, Dolbeault1}, finite-time blow-up solutions in \cite{Nagai1, Nagai2}, and blow-up sets in \cite{Nagai3, Nagai4}. One of the most interesting property of the simplified Keller-Segel system is the critical mass phenomenon, which means that if the mass is strictly less than a certain number, then solutions exist globally and if the mass is strictly bigger than a number, then solutions blow up in finite time. The critical mass was found to be $4\pi$ when $\Omega = B(0,1)$ and $8\pi$ when the initial data are non-negative and radial in \cite{NSY}. However, in higher dimensions, this property is no longer true. In fact, \cite{Winkler-2010} showed that a finite blow-up solution can be constructed in a smooth bounded domain regardless of how small the mass is. \\
Logistic sources \eqref{logistic} describing the growth rate of population \eqref{logistic} can prevent the blow-up solutions. Indeed, it was proved in \cite{Tello+Winkler}, that the term $-\mu u^2$, where $\mu>0$, is sufficient to prevent blow-up phenomena in two dimensional domain. However, in higher dimensions it requires that $\mu$ is larger than a certain number $\mu_0$ when the second equation of \eqref{general-log-1} is elliptic. These results were later improved in \cite{Winkler-2013} for fully parabolic systems and an extra convexity assumption on the domain $\Omega$. Furthermore, a precise formula for $\mu_0$ in a non-convex bounded domain $\Omega \subset \mathbb{R}^3$ was found in \cite{Tian}. A natural question "Is the term $-\mu u^2$ optimal to prevent blow-up in two dimensional domain? " was studied in \cite{Tian1} when $D$ and $S$ are constant functions. It turns out that weaker terms $\frac{-\mu u^2}{\ln^p(u+e)}$ where $0<p<1$ are sufficient to avoid blow-up solutions. \\
In more general conditions for $D$ and $S$, as described by \eqref{diffusion-condition} without the non-decreasing requirement of $S$, were studied in \cite{MW2022}. It was proven that the terms $-\mu u^2$ are sufficient in preventing blow-up solutions by using variational techniques and parabolic regularity theory in Orlicz spaces. In this paper, we apply parabolic regularity results in Orlicz spaces to assert that the term $-\mu u^2$ is not optimal in preventing blow-up solutions. Indeed, our results indicate that solutions to the system \eqref{general-log-1} under the conditions from \eqref{diffusion-condition} to \eqref{boundary-data} exist globally when the terms $-\mu u^2$ are replaced by $\frac{-\mu u^2}{\ln^p(u+e)}$, where $0<p<1$, with an extra assumption that $\inf_{s\geq 0}D(s)> 0$ or $0<p<1/2$ without it.\\
The selling point of the paper is the introduction of the energy functional \[y(t):=\int_\Omega u\ln^k(u+e) +|\nabla v|^2,\] where the value of $k$ is determined later. To establish an appropriate differential inequality for $y$, we perform a tedious analysis calculation, utilize interpolation inequalities in Sobolev spaces, and employ Moser iteration arguments. Our approach is a combination of two previous ideas: the first, proposed in \cite[Lemma 3.2]{Tian1}, offers a common method to obtain a uniform bound for $\left \| u\ln(u) \right \|_{L^1(\Omega)}$, while the second, described in \cite[Lemma 4.5]{MW2022}, provides an additional argument for obtaining a uniform bound for $\left \| u\ln^2(u) \right \|_{L^1(\Omega)}$. It is important to note that using only one of these ideas is insufficient to obtain any $\left \| u\ln^k(u) \right \|_{L^1(\Omega)}$ bounds for solutions. \\
The paper is organized as follows. Section \ref{section-main-theorems} briefly contains our main results. The local well-posedness of solutions and some interpolation inequalities are presented in Section \ref{preliminairies}. In Section \ref{priori}, we establish a priori estimates including $L\ln^k (L+e)$, and $L^2$ bounds for solutions. Finally, the main theorems are proved in Section \ref{prooftheorem}.
\section{Main theorems} \label{section-main-theorems}
In this section, we summarize two main theorems for the existence of global solutions to nondegenerate and degenerate chemotaxis systems. Let us begin with the nondegenate case:
\begin{theorem} [Nondegenerate] \label{nondegenerate}
In addition to the conditions from \eqref{diffusion-condition} to \eqref{boundary-data}, we assume that $p<1$, and $\inf_{s\geq 0} D(s)>0$.
The system \eqref{general-log-1}  possesses a global classical bounded solution at all time.
\end{theorem}
\begin{remark}
Our theorem aligns with and strengthens the outcomes of \cite[Theorem 1.1]{Tian1} by allowing $S$ and $D$ to be arbitrary functions satisfying \eqref{diffusion-condition} rather than being restricted to constant functions.
\end{remark}

The degenerate case are presented as follows:
\begin{theorem}[Degenerate] \label{degenerate}
    If $p<1/2$, then the system \eqref{general-log-1} with the conditions from \eqref{diffusion-condition} to \eqref{boundary-data}  admits a global classical bounded solution in $\Omega \times (0,\infty)$.
\end{theorem}
\begin{remark}
The theorem represents an advancement over the findings of \cite[Theorem 1.4]{MW2022} as it incorporates sub-logistic sources instead of logistic ones. However, it should be noted that our result assumes the non-decreasing property of $S$, whereas \cite[Theorem 1.4]{MW2022} does not require this condition.
\end{remark}
\section{Preliminairies} \label{preliminairies}
The local existence and uniqueness of non-negative classical solutions to the system \eqref{general-log-1} can be established by adapting and adjusting the fixed point argument and standard parabolic regularity theory. For further details, we refer the reader to \cite{Winkler-Horstmann, Tello+Winkler, Lankeit-2017}. For convenience, we adopt Lemma 4.1 from \cite{MW2022}.
\begin{lemma} \label{local-existence}
    Let $\Omega \subset \mathbb{R}^2$ be a bounded domain with smooth boundary, and suppose $r\in \mathbb{R}$ and $\mu>0$ and that \eqref{diffusion-condition}, \eqref{initial-data}, and \eqref{boundary-data} hold. Then there exist $T_{\rm max}\in (0,\infty]$ and functions 
    \begin{equation}
        \begin{cases}
            u \in C^0 \left ( \Bar{\Omega}\times (0,T_{\rm max}) \right ) \cap C^{2,1} \left ( \Bar{\Omega}\times (0,T_{\rm max}) \right )\text{ and} \\
            v \in \bigcap_{q>2} C^0 \left ( [0,T_{\rm max}); W^{1,q}(\Omega
            ) \right )\cap C^{2,1} \left ( \Bar{\Omega}\times (0,T_{\rm max}) \right )
        \end{cases}
    \end{equation}
    such that $u>0$ and $v>0$ in $\Bar{\Omega}\times (0,\infty)$, that $(u,v)$ solves \eqref{general-log-1} classically in $\Omega \times (0,T_{\rm max})$, and that  
    \begin{equation}
        \text{if }T_{\rm max}<\infty, \quad \text{then } \limsup_{t\to T_{\rm max}} \left \{ \left \| u \right \|_{L^\infty(\Omega)} +\left \| u \right \|_{W^{1,\infty}(\Omega)} \right \} = \infty.
    \end{equation}
\end{lemma}
We will use several interpolation inequalities extensively in the following sections. To start, we present an extended version of the Gagliardo-Nirenberg interpolation inequality, which was established in \cite{Li+Lankeit}.
\begin{lemma}[Gagliardo-Nirenberg interpolation inequality ] \label{GN}
Let $\Omega$ be a  bounded and smooth domain of $\mathbb{R}^n$ with $n \geq 1$. Let $r \geq 1$, $0<q\leq p < \infty$, $s>0$. Then there exists a constant $C_{GN}>0$ such that 
\begin{equation*}
    \left \| f \right \|^p_{L^p(\Omega)}\leq C_{GN}\left ( \left \| \nabla f \right \|_{L^r(\Omega)}^{pa}\left \| f \right \|^{p(1-a)}_{L^q(\Omega)} +\left \| f \right \|^p_{L^s(\Omega)}
 \right )
\end{equation*}
for all $f \in L^q(\Omega)$ with $\nabla f \in (L^r(\Omega))^n$, and $a= \frac{\frac{1}{q}-\frac{1}{p}}{\frac{1}{q}+\frac{1}{n}-\frac{1}{r}} \in [0,1]$.
\end{lemma}
Consequently, we have the following lemma:
\begin{lemma} \label{GNY}
If $\Omega$ be a  bounded and smooth domain of $\mathbb{R}^n$ with $n \geq 1$, and $f \in W^{1,2}(\Omega)$ then there exists a positive constant $C$ depending only on $\Omega$ such that the following inequality
\begin{align} \label{GNY1}
    \int_\Omega f^2 \leq C\eta \int_\Omega |\nabla f|^2 + \frac{C}{\eta^{\frac{n}{2}} } \left (   \int_\Omega |f| \right )^2
\end{align}
holds for all $\eta \in (0,1)$.
\end{lemma}
\begin{proof}
The Lemma follows from Lemma \ref{GN} by choosing $p=r=2$ and $q=s=1$ and Young's inequality.
\end{proof}
We will need an interpolation inequality to obtain an $L^q$ bound with $q\geq 2$ for solutions of the system \eqref{general-log-1}. To prove this inequality, we adapt the argument used in the proof of inequality (22) in \cite{Biler+Hebisch}, with some modifications. The following lemma provides a complete proof of this interpolation inequality for the reader's convenience.
\begin{lemma} \label{ine-GN}
    Let $q>0$ and $\Omega \subset \mathbb{R}^2$ be a bounded domain with smooth boundary. Then one can find $C>0$ such that for each $\epsilon >0$, there exists $c(\epsilon )>0$ such that
\begin{equation} \label{unif-GN}
\int_\Omega |u|^{q+1} \leq \epsilon \int_\Omega |\nabla u^{q/2}|^2 \left ( \int_\Omega |u|G(u) \right ) +C \left ( \int_\Omega |u| \right )^{q+1} + c(\epsilon) \int_\Omega |u|
\end{equation}
holds for all $u^{\frac{q}{2}} \in W^{1,2}(\Omega)$, where $G$ is continuous, strictly increasing and  nonnegative  in $[0,\infty)$ such that $\lim_{s\to \infty} G(s) = \infty $. 
\end{lemma}

\begin{proof}
    We call 
    \begin{equation}
     \xi(s)   \left\{\begin{matrix}
 0&  |s| \leq N \\
 2(|s|-N)& N< |s| \leq 2N \\
 |s|& |s|>2N.  \\
\end{matrix}\right.
    \end{equation}
    One can verify that
    \begin{align}
        \int_\Omega ||u|-\xi(u)|^{q+1} \leq (2N)^q \int_\Omega |u|
    \end{align}
and,
\begin{align}
    \int_\Omega \xi(u) \leq \frac{1}{G(N)} \int_\Omega uG(u).
\end{align}
Notice that $\nabla \left (  \xi(u) \right )^{\frac{q}{2}} \leq 4|u|^{q-2}|\nabla u|^2$, combine with Lemma \ref{GN}, we obtain
\begin{align}
    \int_\Omega (\xi(u))^{q+1} &\leq c \int_\Omega |\nabla (\xi(u))^{\frac{q}{2}}|^2 \int_\Omega \xi(u) +C \left(  \int_\Omega \xi(u) \right )^{q+1} \notag \\
    &\leq \frac{c}{G(N)} \int_\Omega |\nabla u^{q/2}|^2 \int_\Omega |u|G(u) +C\left(  \int_\Omega |u| \right )^{q+1}
\end{align}
Hence
\begin{align}
    \int_\Omega |u|^{q+1} &\leq c\left ( \int_\Omega |\xi(u)- |u||^{q+1} +\int_\Omega|\xi(u)|^{q+1} \right ) \notag \\ &\leq \frac{c}{G(N)} \int_\Omega |\nabla u^{q/2}|^2 \int_\Omega |u|G(u) +C\left(  \int_\Omega |u| \right )^{q+1} +(2N)^q \int_\Omega |u|.
\end{align}
We finally complete the proof by choosing N sufficiently large such that $\frac{c}{G(N)} \leq \epsilon $.
\end{proof}
By employing parabolic regularity theory in Sobolev spaces and Orlicz spaces, we can dominate chemical signal terms with density ones. The following lemma states the results in Sobolev spaces, and interested readers can refer to Lemma 2.1 in \cite{Freitag-2018} for additional details.
\begin{lemma} \label{Para-Reg}
Let $p\geq 1$ and $q \geq 1$ satisfy 
\begin{equation*}
    \begin{cases}
     q &< \frac{np}{n-p},  \qquad \text{when } p<n,\\
     q &< \infty, \qquad \text{when } p=n,\\
      q &= \infty, \qquad \text{when } p>n.\\
     \end{cases}
\end{equation*}
Assuming $V_0 \in W^{1,q}(\Omega)$ and $V$ is a classical solution to the following system
\begin{equation}\label{parabolic-equation}
    \begin{cases}
     V_t = \Delta V  - a V + f &\text{in } \Omega \times (0,T), \\ 
\frac{\partial V}{\partial \nu} =  0 & \text{on }\partial \Omega \times (0,T),\\ 
 V(\cdot,0)=V_0   & \text{in } \Omega
    \end{cases}
\end{equation}
where $a>0$ and $T\in (0,\infty]$. If $f \in L^\infty \left ( (0,T);L^p(\Omega) \right ) $, then $V  \in L^\infty \left ( (0,T);W^{1,q}(\Omega) \right )$.
\end{lemma}
The following parabolic regularity result plays a important role in the strongly degenerate case where $\inf_{s\geq 0}D(s) =0$. Indeed, it was proved that equation \eqref{parabolic-equation} possesses a global bounded solution under a suitable slow growth condition of $f$. Precisely, we have the following proposition, which is a direct application of Corollary 1.3 in \cite{MW2022} with $n=2$.
\begin{Prop} \label{Orlicz-regularity}
For each $a>0$, $q>n$, $K>0$ and $\tau>0$, there exist $C(a,q,K,\tau)>0$ such that if $T\geq 2\tau $, $f \in C^0(\Bar{\Omega}\times[0,T])$, and  $V \in C^0 (\Bar{\Omega}\times [0,T]) \cap C^{2,1}(\Bar{\Omega}\times [0,T] ) \cap C^0 ([0,T),W^{1,q}(\Omega))$ are such that \eqref{parabolic-equation} is satisfied with
\begin{align}
    \int_t ^{t+\tau} \int_\Omega |f|^2 \ln^\alpha(|f|+e) <K\text{ for all } t \in ( 0, T-\tau ). 
\end{align}
and  \[\left \|V_0 \right \|_{ W^{1,q}(\Omega)}<K, \]
then
\begin{align}
    |V(x,t)| \leq C(a,q,K,\tau) \quad \text{ for all }(x,t)\in \Omega \times (0,T).
\end{align}
\end{Prop}
The following Lemma is useful in iteration procedure to obtain $L^\infty$ bounds from $L^q$ bounds for some $q>1$.
 \begin{lemma} \label{series}
 Suppose that the positive sequences $(a_k,b_k, u_k)_{k\geq 1}$ satisfy the following conditions:
 \begin{equation}\label{c-l2}
 \begin{cases}
  u_{k+1}\leq a_k+b_ku_k,\\
  \sum_{k=1}^{\infty}a_k =a <\infty,\\
  \prod_{k=1}^\infty b_k =b <\infty,\\
  b_k \geq 1,
 \end{cases}
 \end{equation}
 for all $k \in \mathbb{N}$, then $\sup_k u_k \leq ab+bu_1$.
 \end{lemma}

\begin{proof}
We have 
\begin{align*}
    u_{k+1} &\leq a_k+b_ku_k \leq a_k+a_{k-1}b_k+b_kb_{k-1}u_{k-1}\\
    &\leq a_k +\sum_{i=0}^{k-2} a_{k-1-i}\prod_{j=0}^i b_{k-j} +u_1\prod_{i=1}^k b_i \\
    &\leq b\left ( \sum_{i=1}^k a_i \right ) +bu_1 \leq ab+bu_1.
\end{align*}
\end{proof}
\section{A priori estimates} \label{priori}
In this section, we assume that the system \eqref{general-log-1} admits a classical solution $(u,v)$ and a maximal existence time $T_{\rm max}$, subject to conditions given by \eqref{diffusion-condition} to \eqref{boundary-data}, as established in Lemma \ref{local-existence}. To prove our main theorems, we rely heavily on a bound of the form $L\ln^k(L+e)$ for solutions to the system of equations in \eqref{general-log-1}. The proof utilizes standard variational arguments and fundamental functional inequalities. It is worth noting that the logistic degradation terms in the first equation of \eqref{general-log-1}, given by $-\frac{\mu u^2}{\ln^p(u+e)}$, effectively handle the corresponding cross-diffusion contribution. To precisely state this result, we present the following lemma: 
\begin{lemma} \label{general.lemma1} 
If $p<k<2-p$, then 
   \begin{equation}
       \sup_{t\in (0,T_{\rm max})} \int_\Omega u \ln^k{(u+e)} + |\nabla v|^2 + \sup_{t \in (0, T_{\rm max}-\tau)} \int_t ^{t+\tau} \int_\Omega u^2 \ln^{k-p}(u+e) +(\Delta v)^2 <\infty,
   \end{equation}
   where $\tau = \min \left \{ 1, \frac{T_{\rm max}}{2}\right \}$.
\end{lemma}
\begin{proof}
    We define
 \[
 y(t):= \int_\Omega u \ln^k{(u+e)} + \frac{1}{2}|\nabla v|^2,
 \]
 and differentiate $y(\cdot)$ to obtain
 \begin{align} \label{general.ie.1}
    y'(t)&= \int_\Omega \left ( \ln^k{(u+e)} +k \frac{\ln^{k-1}{(u+e)}}{u+e} \right ) u_t +\nabla v \cdot \nabla v_t  \notag \\
    & := I+J.
 \end{align}
 Now we make use of the first equation of \eqref{general-log-1} to deal with $I$
 \begin{align} \label{general.ie.2}
     I &= \int_\Omega \left ( \ln^k{(u+e)} +k \frac{\ln^{k-1}{(u+e)}}{u+e} \right ) \left ( \nabla \cdot \left ( D(v)\nabla u -u S(v) \nabla v \right ) + f(u) \right ) \notag \\
     &=-k\int_\Omega \frac{D(v)\ln^{k-1}(u+e)}{u+e}|\nabla u|^2 -k(k-1) \int_\Omega \frac{D(v)u \ln ^{k-2}(u+e)}{(u+e)^2} |\nabla u|^2 \notag \\
     &-k \int_\Omega \frac{eD(v)\ln^{k-1}(u+e)}{(u+e)^2 }|\nabla u|^2 +k \int_\Omega\frac{S(v)u \ln ^{k-1}(u+e)}{u+e} \nabla u \cdot \nabla v \notag \\
     &+k(k-1) \int_\Omega \frac{S(v)u^2 \ln ^{k-2}(u+e)}{(u+e)^2}  \nabla u \cdot \nabla v +k\int_\Omega \frac{eS(v)u\ln^{k-1}(u+e) }{(u+e)^2}   \nabla u \cdot \nabla v \notag \\
     &+ \int _\Omega  \left ( \ln^k{(u+e)} +k \frac{\ln^{k-1}{(u+e)}}{u+e} \right ) f(u) \notag \\
     &:= \sum_{i=1}^7 I_i.
 \end{align}
To estimate $I_4$, $I_5$, and $I_6$ from above, we aim to bound them by using two terms $\int_\Omega u^2 \ln ^{k-p}(u+e)$, and  $\int_\Omega (\Delta v)^2$. Achieving this requires a meticulous application of integral by parts and Young's inequality. Specifically, we handle $I_4$ in the following manner:
 \begin{align} \label{general.ie.3}
     I_4 &:= k\int_\Omega \frac{eS(v)u\ln^{k-1}(u+e) }{(u+e)^2}   \nabla u \cdot \nabla v \notag \\
     &= k\int_\Omega S(v) \nabla \phi_1(u) \cdot \nabla v,
 \end{align}
 where \[ \phi_1(u):= \int_0^u \frac{s \ln^{k-1}(s+e)}{s+e} \leq u \ln^{k-1}(u+e). \]
We utilize the integration by parts on equation \eqref{general.ie.3}, taking into account the condition $S' \geq 0$ and applying Young's inequality to obtain
\begin{align} \label{general.ie.4}
    I_4&=-k\int_\Omega S(v) \phi_1(u) \Delta v -k \int_\Omega S'(v) \phi_1(u)|\nabla v|^2 \notag \\
    &\leq c\int_\Omega \phi_1(u)|\Delta v| \notag \\
    &\leq \epsilon  \int_\Omega (\Delta v)^2 +c(\epsilon) \int_\Omega \phi_1^2(u) \notag \\
    & \leq \epsilon \int_\Omega (\Delta v)^2 +c(\epsilon) \int_\Omega u^2 \ln^{2k-2}(u+e) \notag \\
    &\leq \epsilon \int_\Omega (\Delta v)^2 +\epsilon \int_\Omega u^2 \ln^{k-p}(u+e) +c,
\end{align}
 where the last inequality comes from the fact that for any $\delta >0$, there exist a positive constant $c$ depending on $\delta$ such that
 \[
 c(\epsilon)u^2 \ln^{2k-2}(u+e) \leq \delta u^2 \ln^{k-p}(u+e)+c(\delta) \qquad 2k-2<k-p.
 \]
 We apply a similar reasoning to handle $I_5$ and $I_6$. To be more specific, we have:
\begin{align} \label{general.ie.5}
    I_5 &:=k(k-1) \int_\Omega \frac{S(v)u^2 \ln ^{k-2}(u+e)}{(u+e)^2}  \nabla u \cdot \nabla v  \notag \\
    &=k(k-1) \int_\Omega S(v)\nabla \phi_2(u) \cdot \nabla v,
\end{align}
where 
\[
\phi_2(u):= \int_0^u \frac{s^2 \ln ^{k-2}(s+e)}{(s+e)^2} \leq \int_0 ^u \ln ^{k-2}(s+e) \leq u \ln^{k-2}(u+e).
\]
 By using the same procedure to \eqref{general.ie.4}, we obtain 
 \begin{align} \label{general.ie.6}
     I_5&=-k(k-1)\int_\Omega S(v) \phi_2(u) \Delta v -k(k-1) \int_\Omega S'(v) \phi_2(u)|\nabla v|^2 \notag \\
    &\leq c\int_\Omega \phi_2(u)|\Delta v| \notag \\
    &\leq \epsilon  \int_\Omega (\Delta v)^2 +c(\epsilon) \int_\Omega \phi_2^2(u) \notag \\
    & \leq \epsilon \int_\Omega (\Delta v)^2 +c(\epsilon) \int_\Omega u^2 \ln^{2k-4}(u+e) \notag \\
    &\leq \epsilon \int_\Omega (\Delta v)^2 +\epsilon \int_\Omega u^2 \ln^{k-p}(u+e) +c(\epsilon) \quad \left ( 2k-4 <k-p \right).
 \end{align}
  The term $I_6$ can be handled as follows
\begin{align} \label{general.ie.7}
    I_6 &:= k \int_\Omega \frac{euS(v) \ln^{k-1}(u+e)}{(u+e)^2} \nabla u \cdot \nabla v \notag \\
    &= -k\int_\Omega S(v)\nabla \phi_3(u) \cdot \nabla v,
\end{align}
where 
\[
\phi_3(u):= \int_0^u \ln^{k-1}(s+e)\frac{es}{(s+e)^2} \leq  \frac{1}{4}\int_0^u \ln^{k-1}(s+e) \leq u\ln^{k-1}(u+e)
\]
As the right-hand side of \eqref{general.ie.7} resembles that of \eqref{general.ie.3}, we employ the same reasoning to obtain:
\begin{align}\label{general.ie.8}
    I_6 \leq \epsilon \int_\Omega (\Delta v)^2 +\epsilon \int_\Omega u^2 \ln^{k-p}(u+e) +c(\epsilon). 
\end{align}
To handle $I_7$, we make use of the fact that for any $\epsilon>0$, there exist $c(\epsilon)>0$ such that 
\[
u^{a_1}\ln^{b_1}(u+e) \leq \epsilon u^{a_2}\ln^{b_2}(u+e) +c(\epsilon),
\]
where $a_1, a_2, b_1,b_2$ are real numbers such that $a_1 <a_2$. This implies that for any $\epsilon>0$, there exist a positive constant $c$ depending on $\epsilon$ such that
\begin{align}
   \left ( \ln^k{(u+e)} +k \frac{\ln^{k-1}{(u+e)}}{u+e} \right )f(u) &\leq ru\ln^k(u+e)  \notag \\&+rk\ln^{k-1}(u+e)-\mu u^2\ln^{k-p}(u+e) \notag \\
   &\leq (\epsilon-\mu ) \int_\Omega u^2 \ln^{k-p}(u+e) +c(\epsilon).
\end{align}
Therefore, we obtain:
\begin{align}\label{general.ie.9}
    I_7 \leq (\epsilon-\mu ) \int_\Omega u^2 \ln^{k-p}(u+e) +c(\epsilon).
\end{align}
From \eqref{general.ie.2}, \eqref{general.ie.4}, \eqref{general.ie.6}, \eqref{general.ie.8} and \eqref{general.ie.9}, we have
\begin{align} \label{general.ie.10}
    I \leq 3\epsilon \int_\Omega (\Delta v)^2 +(4\epsilon -\mu )\int_\Omega u^2 \ln^{k-p}(u+e) +c(\epsilon).
\end{align}
By integration by parts and elemental inequalities, we handle $J$ as follows:
\begin{align} \label{general.ie.11}
    J &:= \int_\Omega \nabla v \cdot \nabla v_t \notag \\
    &=-\int_\Omega (\Delta v)^2 -\int_\Omega|\nabla v|^2 -\int_\Omega u\Delta v \notag \\
    &\leq -\frac{1}{2}\int_\Omega (\Delta v)^2 -\int_\Omega|\nabla v|^2  +\frac{1}{2 }\int_\Omega u^2 \notag \\
    &\leq -\frac{1}{2}\int_\Omega (\Delta v)^2 -\int_\Omega|\nabla v|^2  +\epsilon \int_\Omega u^2 \ln^{k-p}(u+e)+c(\epsilon).
\end{align}
For any $\epsilon>0$, there exist a positive constant $c(\epsilon)$ such that
\begin{align} \label{general.ie.12}
    \int_\Omega u \ln^{k}(u+e) \leq \epsilon \int_\Omega u^2 \ln^{k-p}(u+e)+c(\epsilon).
\end{align}
By combining \eqref{general.ie.1}, \eqref{general.ie.10}, \eqref{general.ie.11}, and \eqref{general.ie.12}, we obtain that for any $\epsilon>0$, there exist a positive constant $c$ depending on $\epsilon$ such that
\newpage
\begin{align}\label{general.ie.13}
    y'(t) +y(t) +\frac{1}{4}\int_\Omega (\Delta v)^2 +\frac{\mu}{2} \int_\Omega u^2 \ln^{k-p}(u+e) &\leq (3\epsilon-\frac{1}{4})\int_\Omega (\Delta v)^2 \notag \\  &+ (6\epsilon -\frac{\mu}{2})\int_\Omega u^2 \ln^{k-p}(u+e)+c,
\end{align}
Choose $\epsilon$ sufficienly small, we have 
\[
y'(t)+y(t) \leq c.
\]
Using Gronwall's inequality with the previous equation, we can conclude that $y(t)\leq \max \left \{ y(0), c \right \}$.
 Additionally, we also have:
\begin{align}\label{general.ie.14}
    \frac{1}{4}\int_\Omega (\Delta v)^2 +\frac{\mu}{2} \int_\Omega u^2 \ln^{k-p}(u+e) \leq c -y'(t).
\end{align}
By integrating the previous inequality from $t$ to $t+\tau$ and using the fact that $y$ is bounded, we can conclude the proof.
\end{proof}
\begin{remark}
  The non-decreasing assumption of $S$ allows us to obtain a uniform bound for $\left \| u\ln^k{(u+e)}  \right \|_{L^1(\Omega)}$ without using a uniform bound $\left \| u \right \|_{L^1(\Omega)}$ as in \cite{MW2022} and \cite{Tian1}.
\end{remark}
The logistic degradation term $-\frac{-\mu u^2}{\ln^p(u+e)}$ can ensure the boundedness of chemical density functions, even in the presence of strongly degenerate diffusion terms. To state this result precisely, we present the following lemma.
\begin{lemma} \label{general.lemma2}
    If $1+p<k$, and
    \[
    \sup_{t \in (0,T_{\rm max}-\tau)} \int_t^{t+\tau} \int_\Omega u^2 \ln^{k-p}(u+e) < \infty,
    \]
    where  $\tau = \min \left \{ 1, \frac{T_{\rm max}}{2}\right \}$, then $v$ is globally bounded in time.
\end{lemma}
\begin{proof}
    This is a direct application of Proposition \ref{Orlicz-regularity} with $\alpha =k-p>1$.
\end{proof}
We examine the nondegenerate diffusion mechanism and obtain bounds for $u$ and $\nabla v$ through a standard testing procedure.
\begin{lemma} \label{L2est.E1}
If $p<1$, $q\geq 2$, $S' \geq 0$, $\inf_{s\geq 0}D(s)>0$ and $(u,v)$ is a classical solution to \eqref{general-log-1} in $\Omega \times (0,T_{\rm max})$ then there exists a positive constant $C$ such that
\begin{equation} \label{L2estimate}
    \int_\Omega u^q(\cdot,t) + \int_\Omega |\nabla v(\cdot,t)|^{2q} \leq C
\end{equation}
for all $t \in (0, T_{\rm max})$.
\end{lemma}
\begin{proof}
   We define 
\begin{equation*}
    \phi(t):= \frac{1}{q} \int_\Omega u^q +\frac{1}{2q} \int_\Omega |\nabla v|^{2q},
\end{equation*}
and differentiate $\phi$ to obtain:
\begin{align} \label{L2ets1}
    \phi'(t)&= \int_\Omega u^{q-1}\left [ \nabla \cdot (D(v)\nabla u) - \nabla \cdot (S(v)u \nabla v)  +f(u)  \right ] \notag \\
    &+\int_\Omega |\nabla v|^{2q-2} \nabla v \cdot \nabla  \left ( \Delta v +u- v \right ) \notag \\
    &:= J_1+J_2.
\end{align}
By integration by parts, we have
\begin{align} \label{L2ets2}
    J_1 &=- c\int_\Omega D(v) |\nabla u^{\frac{q}{2}}|^2 +  c\int_\Omega S(v)u^{\frac{q}{2}} \nabla u^{\frac{q}{2}} \cdot \nabla v +r\int_\Omega u^q -\mu\int_\Omega \frac{u^{q+1} }{\ln^p(u+e)} \notag\\
    &:= J_{11}+J_{12}+J_{13}+J_{14}.
\end{align}
Since $\inf_{(x,t)\in \Omega \times (0,T)} D(v(x,t)) >0$, we obtain
\begin{align}
    J_{11} \leq -c\int_\Omega |\nabla u^{\frac{q}{2}}|^2,
\end{align}
for some $c>0$. For any $\epsilon>0$, there exist a positive constant $c$ depending on $\epsilon$ such that
\begin{align} \label{L2ets3}
    J_{12} \leq  \epsilon \int_\Omega |\nabla u^{\frac{q}{2}}|^2 +c\left \| S \right \|_{{L^\infty}(0,\infty)} \int_\Omega u^q |\nabla v|^2.
\end{align}
Choosing $\epsilon$ sufficiently small implies that
\begin{align} \label{J1-estimate}
    J_1 \leq -c\int_\Omega |\nabla u^{\frac{q}{2}}|^2 +c\int_\Omega u^q |\nabla v|^2+r\int_\Omega u^q -\mu\int_\Omega \frac{u^{q+1} }{\ln^p(u+e)} 
\end{align}
In treating $J_2$, we make use of the following pointwise identity
\begin{align*}
    \nabla v \cdot \nabla \Delta v = \frac{1}{2} \Delta( |\nabla v|^2) - |D^2v|^2 
\end{align*}
to obtain
\begin{align} \label{L2ets6}
    J_2&= -c\int_\Omega |\nabla|\nabla v|^q|^2-\int_\Omega |\nabla v|^{2q-2} |D^2 v|^2 \notag\\
    &+\int_\Omega |\nabla v|^{2q-2}\nabla v\cdot \nabla u \notag \\
    &- \int_\Omega |\nabla v|^{2q} +c\int_{\partial \Omega} \frac{\partial |\nabla v|^2 }{\partial \nu}|\nabla v|^{2q-2}.
\end{align}
The inequality $\frac{\partial |\nabla v|^{2q-2}}{\partial \nu} \leq c |\nabla v|^2 $ for some $c>0$ depending only on $\Omega$ implies that
\[
\int_{\partial \Omega} \frac{\partial |\nabla v|^2 }{\partial \nu}|\nabla v|^{2q-2}\, dS \leq c\int_{\partial \Omega }|\nabla v|^{2q} \, dS  .
\]
By Trace Imbedding Theorem and Young's inequality, we obtain
\begin{align} \label{nonconvex-1}
    \int_{\partial \Omega }|\nabla v|^{2q} \, dS \leq \epsilon \int_{\Omega}|\nabla|\nabla v|^q|^2 +c(\epsilon) \int_\Omega |\nabla v|^{2q}.  
\end{align}
 Applying the pointwise inequality $(\Delta v)^2 \leq 2 |D^2 v|^2$ to \eqref{L2ets6} yields
 \newpage
\begin{align} \label{L2ets7}
    J_2  &\leq   -c\int_\Omega |\nabla|\nabla v|^{q}|^2 -\frac{1}{2}\int_\Omega |\nabla v|^{2q-2} |\Delta v|^2 \notag\\
    &+\int_\Omega |\nabla v|^{2q-2}\nabla v\cdot \nabla u +c\int_\Omega |\nabla v|^{2q} \notag \\
    &=J_{21}+J_{22}+J_{23} +J_{24}.
\end{align}
By integration by parts and elemental inequalities, we obtain that for any $\epsilon>0$, there exist a positive constant $c$ depending on $\epsilon$ such that
\begin{align}
    J_{23}=\int_\Omega|\nabla v|^{2q-2}\nabla v \cdot \nabla u 
    &\leq \epsilon \int_\Omega (\Delta v)^2 |\nabla v|^{2q-2} + \epsilon\int_\Omega |\nabla |\nabla v|^q|^2  \notag \\
    &+ c\int_\Omega u^2 |\nabla v|^{2q-2}. 
\end{align}
Choosing $\epsilon$ sufficiently small, we obtain
\begin{align} \label{J2-estimate}
    J_2 \leq -c\int_\Omega |\nabla|\nabla v|^{q}|^2+c\int_\Omega |\nabla v|^{2q}+c\int_\Omega u^2 |\nabla v|^{2q-2}.
\end{align}
By Young inequality, we have
\begin{align} \label{L2ets11}
  c \int_\Omega u^q|\nabla v|^{2}+c \int_\Omega u^2|\nabla v|^{2q-2}\leq   c\epsilon \int_\Omega |\nabla v|^{2q+2}+ c(\epsilon)\int_\Omega u^{q+1}.
\end{align}
Using the Gagliardo-Nirenberg inequality in Lemma \ref{GN} for $n=2$ and Lemma \ref{general.lemma1}, we can conclude that there exists a positive constant $c_{GN}$ such that:
\begin{align} \label{L2ets12}
    \int_\Omega |\nabla v|^{2q+2} &\leq c_{GN}\int_\Omega |\nabla |\nabla v|^q|^2  \int_\Omega  |\nabla v|^2 + c_{GN}\left (\int_\Omega  |\nabla v|^2 \right )^{q+1} \notag \\
    &\leq c\int_\Omega |\nabla |\nabla v|^q|^2  + c.
\end{align}
The condition $0<p<1$ enables us to choose $k\in (p,2-p)$, particularly we select $k=1$ and apply Lemma \ref{general.lemma1} to obtain  the uniformly boundedness of $\left \| u\ln(u+e) \right \|_{L^1(\Omega)}$. This together with Lemma \ref{ine-GN} imply that for any $\epsilon>0$, there exist a positive constant $c$ depending on $\epsilon$ satisfying
\begin{align} \label{L2ets14}
    \int_\Omega u^{q+1} &\leq \epsilon  \int_\Omega |\nabla u^{\frac{q}{2}}|^2   \int_\Omega u \ln(u+e) +c \left ( \int_\Omega u \right )^{q+1} +c \notag\\
  &\leq c\epsilon \int_\Omega |\nabla u^{\frac{q}{2}}|^2+c.
\end{align}
Combining \eqref{L2ets1}, \eqref{J1-estimate}, and from \eqref{J2-estimate} to \eqref{L2ets14}, and choosing $\epsilon$ sufficiently small, we obtain
\begin{align} \label{phi-equation}
 \phi'(t)\leq -c\int_\Omega |\nabla|\nabla v|^{q}|^2 +c\int_\Omega |\nabla v|^{2q}+r \int_\Omega u^q  -\mu \int_\Omega \frac{u^{q+1}}{\ln^p(u+e)}+c.   
\end{align}
For any $\epsilon>0$, there exist a positive constant $c$ depending on $\epsilon$ such that
\[
x^q \leq \frac{\epsilon x^{q+1}}{\ln^p(x+e)}+c.
\]
This implies that 
\begin{align} \label{est-for-u}
    \int_\Omega u^q \leq \epsilon \int_\Omega \frac{u^{q+1}}{\ln^p(u+e)}+c
\end{align}
By applying Lemma \ref{GN} and using the fact that $\left \| \nabla v\right \|_{L^2(\Omega)}$ is uniformly bounded, and combining with Young inequality we obtain that for any $\epsilon>0$, there exist a positive constant $c$ depending on $\epsilon$ such that
\begin{align} \label{est-for-gradiant v}
    \int_\Omega |\nabla |\nabla v|^q|^2 &\leq c_{GN} \left ( \int_\Omega |\nabla |\nabla v|^q|^2 \right )^{\frac{q-1}{q}} \int_\Omega |\nabla v|^2 + \left ( \int_\Omega |\nabla v|^2 \right )^q \notag \\
    &\leq c \left ( \int_\Omega |\nabla |\nabla v|^q|^2 \right )^{\frac{q-1}{q}} +c \notag \\
    &\leq \epsilon \int_\Omega |\nabla |\nabla v|^q|^2 +c
\end{align}
By combining \eqref{phi-equation}, \eqref{est-for-u}, and \eqref{est-for-gradiant v}, and selecting an appropriate value for $\epsilon$, we can conclude that $\phi'(t) + \phi(t) \leq c$. The proof is completed by applying Gronwall's inequality.
\end{proof}
When the chemical concentration function $v$ is bounded, the degeneracies in the diffusion mechanism are eliminated, thus enabling us to derive bounds for $u$ and $\nabla v$. Specifically, we present the following lemma.
\begin{lemma} \label{L2est.E1'}
If $p<1/2$, $q\geq 2$, $S' \geq 0$, and $(u,v)$ is a classical solution to \eqref{general-log-1} in $\Omega \times (0,T)$ then there exists a positive constant $C$ such that
\begin{equation} \label{L2estimate'}
    \int_\Omega u^{q}(\cdot,t) + \int_\Omega |\nabla v(\cdot,t)|^{2q} \leq C
\end{equation}
for all $t \in (0, T_{\rm max})$.
\end{lemma}
\begin{proof}
     Since $0<p<\frac{1}{2}$, we can select a constant $k \in (1+p, 2-p)$. By utilizing Lemma  \ref{general.lemma1}, we obtain 
    \[
    \sup_{t \in (0, T-\tau)} \int_t ^{t+\tau } \int_\Omega u^2 \ln^{k-p}(u+e) <\infty.
    \]
    Then, applying Lemma \ref{general.lemma2}, we deduce that $v$ is globally bounded in time, implying that $D(v) \geq c>0$. Using the same argument as in the proof of Lemma \ref{L2est.E1}, we can conclude the proof.
\end{proof}
It is possible to obtain an $L^\infty$ bound for solutions of equation \eqref{general-log-1} by using Lemma A.1 in \cite{Winkler-2011}, provided that we have $L^{q_0}$ bounds for some $q_0 > 2$. However, for the sake of completeness, we present a proof that uses the Moser iteration method \cite{Alikakos1, Alikakos2} to establish the iteration process from $L^{q_0}$ to $L^\infty$. To this end, we rely on the following lemma: 
\begin{lemma}\label{bddu}
Let $(u,v)$ be a classical solution of \eqref{general-log-1} on $(0,T_{\rm max})$ and \[ U_q := \max \left \{ \| u_0 \|_{L^\infty(\Omega)}, \sup_{t \in (0, T_{\rm max})} \| u(\cdot, t )  \|_{L^q{\Omega)}}  \right \}.   \] 
If $\sup_{t \in (0, T_{\rm max})} \| u(\cdot, t )  \|_{L^{q}{\Omega)}} < \infty$ for some   $q>n$, then there exists constants $A,B>0$ independent of $q$ such that 
 \begin{align} \label{bddu-ine}
    U_{2q} \leq  (Aq^{B})^{\frac{1}{2q}}U_q.  
 \end{align}
\end{lemma} 
\begin{proof}
The primary objective is to initially establish an inequality of the form:
\begin{align} \label{ine-prop}
\frac{d}{dt}\int_{\Omega} u^{2q} +\int_{\Omega} u^{2q} \leq Aq^B \left ( \int_{\Omega} u^{q} \right )^2,
\end{align}
where $A$ and $B$ are positive constants. We then proceed to apply the Moser iteration technique. It is crucial to note that the dependence of all the constants on $q$ is tracked carefully. Multiplying the first equation in the system \eqref{general-log-1} by $u^{2q-1}$ we obtain
\begin{align} \label{bddu.1}
    \frac{1}{2q}\frac{d}{dt}\int_{\Omega} u^{2q} &= \int_{\Omega} u^{2q-1}u_t \notag\\
    &=  \int_\Omega u^{2q-1} \left [ \nabla \cdot (D(v) \nabla u) - \nabla \cdot ( S(v)u \nabla v)  +r u -\frac{\mu u^2}{\ln^p(u+e)} \right ] \notag \\
    &:=I +J+K.
\end{align}
Since there exist $C>0$ such that $\int_\Omega u^q (\cdot,t )< C$ for all $t\in (0, T_{\rm max})$, Lemma \ref{Para-Reg} entails that $v$ is globally bounded, which further implies $\inf_{(x,t)\in \Omega \times (0.T_{\rm max})}D(v(x,t)):=c_1 >0$. Thus, we have
\begin{align}
   I:= -\frac{2q-1}{q^2}\int_{\Omega} D(v) |\nabla u^q|^2 \leq -c_1\frac{2q-1}{q^2} \int_{\Omega}  |\nabla u^q|^2. 
\end{align}
In treating $J$,
\begin{align} \label{bddu.2}
     J &:= \int_\Omega u^{2q-1} \nabla \cdot ( S(v)u \nabla v)\notag\\
    &= \chi \frac{2q-1}{2q} \int_\Omega S(v)\nabla u^{2q} \cdot \nabla v  \\
    &= \chi  \frac{2q-1}{q} \int_\Omega S(v)u^{q} \nabla u^{q}\cdot \nabla v 
\end{align}
 Lemma \eqref{Para-Reg} asserts that $v$ is in $L^\infty\left ( (0,T); W^{1,\infty}(\Omega) \right )$. Thus,
\begin{equation*}
    \sup_{0<t<T}  \left \| \nabla v \right \|^2_{L^\infty}  = c_2 <\infty,
\end{equation*}
 Apply Young inequality  yields 
\begin{align} \label{bddu.3}
    J &\leq \epsilon \int_{\Omega} |\nabla u^q|^2 + \frac{(2q-1)^2}{4q^2 \epsilon} \left \| S \right \|_{L^\infty(0,\infty)} \int _{\Omega} u^{2q} |\nabla v|^2  \notag \\
    &\leq \epsilon \int_{\Omega} |\nabla u^q|^2 + \frac{(2q-1)^2}{4q^2 \epsilon} \left \| S \right \|_{L^\infty(0,\infty)}c_2 \int _{\Omega} u^{2q}.  
\end{align}
It follows from \eqref{bddu.1} and \eqref{bddu.3} that
\begin{align}\label{bddu.5}
   \frac{d}{dt} \int_\Omega u^{2q}+\int_\Omega u^{2q} & \leq 2q\left ( -\frac{2q-1}{q^2} +\epsilon \right ) \int_{\Omega} |\nabla u^{q}|^2 - 2q\mu  \int_\Omega u^{2q+1} \notag \\  
  &+ \left [ \frac{(2q-1)^2}{2q \epsilon}\chi^2c_2 \left \| S \right \|_{L^\infty(0,\infty)} + 2qr+1 \right ] \int_{\Omega} u^{2q}.
\end{align}
Substitute $\epsilon=\min \left \{ \frac{q-1}{q^2}, \mu  \right \} $ into \eqref{bddu.5}  we obtain
\begin{align} \label{*}
    \frac{d}{dt}\int_\Omega u^{2q} + \int_\Omega u^{2q} \leq -2  \int_{\Omega} |\nabla u^{q}|^2+c_3q^2 \int_\Omega u^{2q}
\end{align}
where  $c_3$ are independent of $q$. Apply Lemma \ref{GNY}, and plug into \eqref{*} entails the following inequality  for all $\eta \in (0,1)$
\begin{align}
    \frac{d}{dt}\int_\Omega u^{2q} + \int_\Omega u^{2q} \leq (c_3q^2\eta -2)\int_{\Omega} |\nabla u^{q}|^2+ \frac{c_4q^2}{\eta^{\frac{n}{2}}} \left ( \int_{\Omega} u^q \right )^2,
\end{align}
 where $c_4>0$ independent of $r,\eta$. Substitute $\eta = \min \left \{  \frac{1}{c_3q^2}, 1 \right \}$ into this yields
\begin{align} \label{**}
    \frac{d}{dt}\int_\Omega u^{2q} + \int_\Omega u^{2q} \leq c_5q^{n+2}\left ( \int_{\Omega} u^q \right )^2,
\end{align}
where $c_5$ independent of $q$. Apply Gronwall inequality yields 
\[
\int_\Omega u^{2q}(\cdot,t)  \leq \max \left \{ c_5 q^{n+2} U_q^{2q} ,\int_\Omega u_0^{2q}    \right \}
\]
This entails
\begin{align*}
    \left \| u(\cdot, t)  \right \|_{L^{2q}(\Omega)} \leq \max \left \{ (c_5q^{n+2})^\frac{1}{2q} U_q,  |\Omega|^{\frac{1}{2q}}  \left \| u_0 \right \|_{L^{\infty}(\Omega)}  \right \},
\end{align*}
and further implies that 
\begin{align*}
    U_{2q} \leq (A q^{B})^\frac{1}{2q} U_q
\end{align*}
where ${A} = \max \left \{ c_5, |\Omega| \right \}$ and $B=n+2$. The proof of \eqref{bddu-ine} is complete.
\end{proof}
\section{Proof of main theorems}\label{prooftheorem}
This section focuses on proving our main theorems, starting with the non-degenerate case.
\begin{proof}[Proof of Theorem \ref{nondegenerate}]
    From Lemma  \ref{general.lemma1} and  Lemma \ref{L2est.E1}, for some fixed $q_0>2$ 
    \begin{align}
        \sup_{t\in (0,T_{\rm max})} \int_\Omega u^{q_0} +|\nabla v|^{2q_0} \leq C<\infty.
    \end{align}
By using Lemma \ref{Para-Reg}, we can conclude that $v$ belongs to $L^\infty \left ( (0,T_{\rm max}); W^{1,\infty}(\Omega) \right )$. Furthermore, Lemma \ref{bddu} implies that the following inequality holds
\begin{align} \label{f}
    U_{2^{k+1}q_0} \leq \left (  A(2^{k}q_0)^B \right )^{\frac{1}{2^{k+1}q_0}}U_{2^kq_0}
\end{align}
for all integers $k \geq 0$. After taking the $\log$ of the above inequality, we can use Lemma \ref{series} for the following sequence.
  \begin{align*}
      a_k&= \frac{\ln A}{2^{k+1}q_0}+ \frac{Bk \ln 2}{2^{k+1} q_0} + \frac{B \ln q_0}{2^{k+1}q_0}
  \end{align*}
One can verify that 
\begin{align*}
    \sum_{k=0}^{\infty} a_k = \frac{\ln \left (A(2q_0)^B  \right ) }{q_0}.
\end{align*}
Thus, we obtain
\begin{equation}
     U_{2^{k+1}q_0}\leq A^\frac{1}{q_0}(2q_0)^{\frac{B}{q_0}}U _{q_0} 
  \end{equation}
  for all $k\geq 1$. Send $k \to \infty$ yields
  \begin{equation}  \label{uinfty}
       U_{\infty}\leq  A^\frac{1}{q_0}(2q_0)^{\frac{B}{q_0}}U _{q_0}. 
  \end{equation}
 This implies that $u \in L^\infty \left ( (0,T_{\rm max}); L^\infty (\Omega) \right )$.
\end{proof}
The proof of Theorem \ref{degenerate} is similar to that of Theorem \ref{nondegenerate}, with the additional requirement of showing that the diffusion mechanism remains non-degenerate throughout the evolution of the system.
\begin{proof}[Proof of Theorem \ref{degenerate}]
    By using Lemma \ref{general.lemma2} and Lemma \ref{L2est.E1'}, it follows that for a fixed $q_0>2$, we have
    \begin{align}
        \sup_{t\in (0,T_{\rm max})} \int_\Omega u^{q_0} +|\nabla v|^{2q_0} \leq C<\infty.
    \end{align}
   We can now repeat the same arguments from \eqref{f} to \eqref{uinfty} to establish $L^\infty$ bounds for $u$ and $v$.
\end{proof}

\section*{Acknowledgement}
The writer is indebted to Professor Michael Winkler for his kindly assistance in providing insightful comments, suggestions and valuable references. The writer also would like to thank Prof. Zhengfang Zhou his continuous encouragement and support.
\printbibliography
\end{document}